\documentclass[12pt]{article}

\usepackage{amsmath}
\usepackage{amsthm}
\usepackage{amssymb}
\usepackage{hyperref}

\newtheorem{theorem}{Theorem}[section]

\newtheorem{definition}[theorem]{Definition}
\newtheorem{proposition}[theorem]{Proposition}
\newtheorem{corollary}[theorem]{Corollary}
\newtheorem{lemma}[theorem]{Lemma}
\newtheorem{remark}[theorem]{Remark}

\newcommand{\bdfn}{\begin{definition}}
\newcommand{\edfn}{\end{definition}}
\newcommand{\bthm}{\begin{theorem}}
\newcommand{\ethm}{\end{theorem}}
\newcommand{\bprop}{\begin{proposition}}
\newcommand{\eprop}{\end{proposition}}
\newcommand{\bcor}{\begin{corollary}}
\newcommand{\ecor}{\end{corollary}}
\newcommand{\blem}{\begin{lemma}}
\newcommand{\elem}{\end{lemma}}

\newcommand{\be}{\begin{enumerate}}
\newcommand{\ee}{\end{enumerate}}
\newcommand{\bt}{\begin{tabular}}
\newcommand{\et}{\end{tabular}}
\newcommand{\beq}{\begin{equation}}
\newcommand{\eeq}{\end{equation}}
\newcommand{\ba}{\begin{array}}
\newcommand{\ea}{\end{array}}
\newcommand {\bea} {\begin{eqnarray}}
\newcommand {\eea} {\end {eqnarray}}
\newcommand {\bua} {\begin{eqnarray*}}
\newcommand {\eua} {\end {eqnarray*}}

\providecommand{\defnterm}[1]{{\em #1}}
\newcommand{\Ra}{\Rightarrow}
\newcommand{\se}{\subseteq}
\newcommand{\ol}{\overline}
\newcommand{\ds}{\displaystyle}

\def\R{{\mathbb R}}
\def\N{{\mathbb N}}
\def\Z{{\mathbb Z}}
\newcommand{\cO}{\mathcal{O}}

\newcommand{\eps}{\varepsilon}

\newcommand{\limn}{\ds\lim_{n\to\infty}}
\newcommand{\Dlimn}{\ds\Delta-\lim_{n\to\infty}}
\newcommand{\lsupn}{\ds \limsup_{n\to\infty}}

\newcommand{\lambdaxy}{(1-\lambda)x\oplus\lambda y}
\newcommand{\lambdaxz}{(1-\lambda)x\oplus\lambda z}

\newcommand{\muyw}{(1-\mu)y\oplus\mu w}
\newcommand{\alphaxw}{(1-\alpha)x\oplus\alpha w}
\newcommand{\tlambdaxy}{(1-\tilde{\lambda})x\oplus\tilde{\lambda} y}

\usepackage[text={6.5in,8in},centering]{geometry}
\newcommand{\conv}[1]{\xrightarrow{\;#1\;}}

\newcommand{\comp}{{\circ}}
\DeclareMathOperator*{\argmin}{argmin}

\begin{document}

\title{Firmly nonexpansive mappings in classes of geodesic spaces}
\author{David Ariza-Ruiz$^{1}$, Lauren\c tiu Leu\c stean$^{2}$, Genaro L\'{o}pez-Acedo$^{1}$\\[0.2cm]
\footnotesize ${}^1$ Dept. An\'alisis Matem\'atico, Fac. Matem\'aticas,\\
\footnotesize Universidad de Sevilla, Apdo. 1160, 41080-Sevilla, Spain\\[0.1cm]
\footnotesize${}^2$ Simion Stoilow Institute of Mathematics of the Romanian Academy, \\
\footnotesize P. O. Box 1-764, RO-014700 Bucharest, Romania\\[0.1cm]
\footnotesize E-mails: dariza@us.es, Laurentiu.Leustean@imar.ro, glopez@us.es
}

\date{}
\maketitle

\begin{abstract}
Firmly nonexpansive mappings play an important role in metric fixed point theory and optimization due to their
correspondence with maximal monotone operators. In this paper we do a thorough study of fixed
point theory and the  asymptotic behaviour of Picard iterates of these mappings in different classes of geodesic spaces, such as
(uniformly convex) $W$-hyperbolic spaces, Busemann spaces and CAT(0) spaces. Furthermore, we apply methods of
proof mining to obtain effective rates of asymptotic regularity for the Picard iterations.  \\[0.1cm]

\noindent {\em MSC: } Primary: 47H09, 47H10, 53C22; Secondary: 03F10, 47H05, 90C25, 52A41.\\

\noindent {\em Keywords}: firmly nonexpansive mappings, geodesic spaces, uniform convexity, Picard iterates, asymptotic regularity,
$\Delta$-convergence, proof mining, effective bounds, minimization problems.
\end{abstract}

\section{Introduction}

Let $C$ be a closed convex subset of a Hilbert space $H$. {\em Firmly contractive} mappings were defined by
Browder \cite{Bro67} as mappings $T:C\to H$ satisfying the following inequality for all $x,y\in C$:
\begin{equation}\label{int:eq:1}
    \|Tx-Ty\|^2\leq \langle x-y,Tx-Ty\rangle.
\end{equation}
As Browder points out, these mappings play an important role in the study of (weak) convergence for sequences of
nonlinear operators. An example of a firmly contractive mapping is the metric projection $P_C:H\to H$, defined by
$ P_C(x)=\argmin_{y\in C}\{\Vert x-y\Vert\}$.
One can easily see that any firmly contractive mapping $T$ is nonexpansive, i.e. satisfies
$\|Tx-Ty\|\leq \|x-y\|$ for all $x,y\in C$.  The converse is not true, as one can see by taking $T=-Id$.
\par\medskip
In his study  of  nonexpansive projections on subsets of Banach spaces, Bruck \cite{Bru73} defined a
{\em firmly nonexpansive} mapping  $T:C\to E$, where $C$ is a closed convex subset of a real
Banach space $E$, to be a mapping with the property that for all $x,y\in C$ and $t\geq 0$,
\begin{equation}\label{int:eq:2}
    \|Tx-Ty\|\leq\|(1-t)(Tx-Ty)+t(x-y)\|.
\end{equation}
In Hilbert spaces these mappings coincide with the firmly contractive ones introduced by Browder.
As Bruck shows, to any nonexpansive selfmapping $T:C\to C$ that has fixed points, one can associate
a 'large' family of firmly nonexpansive mappings having the same fixed point set with $T$.
Hence, from the point of view  of the existence of fixed points on convex closed sets, firmly nonexpansive
mappings  exhibit a similar behaviour with the nonexpansive ones. However, this is not anymore true if we
consider non-convex domains \cite{Sma91}. Firmly nonexpansive mappings in Banach spaces have also
been studied in \cite{BruRei77} and \cite{Rei77}.
\par\medskip
If $T$ is firmly nonexpansive and has fixed points, it is well known \cite{Bro67} that
the Picard iterate $(T^nx)$ converges weakly to a fixed point of $T$ for any starting point $x$, while this is
not true for nonexpansive mappings (take again $T=-Id$). This is a first reason for the importance of
firmly nonexpansive mappings.
\par\medskip
A second reason for the importance of this class of mappings  is their correspondence with maximal
monotone operators, due to Minty \cite{Min62}.

The resolvent of a monotone operator was introduced by Minty \cite{Min62} in Hilbert spaces and by
Br\'{e}zis, Crandall and Pazy \cite{BreCraPaz70} in Banach spaces. Among other applications, the resolvent
has proved to be very useful in the study of the asymptotic behaviour of the solutions of the Cauchy abstract problem governed
by a monotone operator, see for instance \cite{GarRei06,NevRei79,Xu01}. Given a maximal monotone
operator $A:H\to 2^H$ and $\mu>0$, its associated {\em resolvent}
of order $\mu$, defined by $J_\mu^A:=(Id+\mu A)^{-1}$, is a firmly nonexpansive mapping from
$H$ to $H$ and the set of fixed points of $J_\mu^A $ coincides with the set of zeros of $A$. We refer to \cite{BauMofWan12} for a
very nice presentation of this correspondence. Rockafellar's \cite{Roc76} proximal point algorithm uses
the resolvent to approximate the zeros of maximal monotone operators.
\par\medskip
The subdifferential of a proper, convex and lower semicontinuous function $F:H\to(-\infty,\infty]$ is a
maximal monotone operator, hence  the resolvent associated to the subdifferential is a
firmly nonexpansive mapping, that coincides with the proximal map introduced by Moreau \cite{Mor65}.
The proximal point algorithm for approximating the minimizers of $F$ is based on the weak convergence towards a
fixed point of the Picard iterate of the resolvent and the fact that the minimizers of $F$ are the fixed points
of the resolvent.
\par\medskip
In the last 20 years a fruitful direction of research consists of extending techniques and results obtained in
normed spaces to metric spaces without linear structure. For instance  minimization problems associated to
convex functionals have been solved in the setting of Riemannian manifolds \cite{FerPerNem05,LiLopNar09}, while
some problems have been modelled as abstract Cauchy equations  in the framework  of nonpositive curvature
geodesic metric spaces (see \cite{May98,Sto11} and references therein).
Although apparently the framework and the conceptual approach in the previous problems are
quite different, it is possible, as in the case of normed spaces, to find a bridge between them through
firmly nonexpansive mappings.
\par\medskip
{\em The goals of our work are twofold. First we generalize known results on firmly nonexpansive mappings
in Hilbert or Banach spaces to suitable classes of geodesic spaces. Second we obtain effective results on
the asymptotic behaviour of Picard iterations.}
\par\medskip
In Section \ref{cap:geodesic} we give basic definitions and properties of the classes of geodesic spaces we
consider in this paper: $W$-hyperbolic spaces, $UCW$-hyperbolic spaces, Busemann spaces and CAT(0) spaces.
We recall properties of asymptotic centers in such spaces, that are essential for our results.

Firmly nonexpansive mappings in the Hilbert ball and, more generally, in hyperbolic spaces, have already been studied in
\cite{GoeRei84,ReiSha87,ReiSha90} and, more recently, in the paper by Kopeck\'{a} and Reich \cite{KopRei09}.
In Section \ref{cap:firmly-ne} we extend Bruck's definition of firmly nonexpansive mapping to our class of $W$-hyperbolic spaces.
We show that, in the setting of CAT(0) spaces, the metric projection on a closed
convex set and the resolvent of a proper, convex and lower semicontinuous mapping are firmly nonexpansive.
Furthermore, Bruck's association of a family of firmly nonexpansive mappings to any
nonexpansive mapping is adapted to Busemann spaces.

Section \ref{cap:fpt-theorem} contains a fixed point theorem for firmly nonexpansive mappings defined on
finite unions of closed convex subsets of a complete $UCW$-hyperbolic space. Our result generalizes and
strengthens Smarzewski's \cite{Sma91} fixed point theorem for uniformly convex Banach spaces.
In this section we also obtain new results about periodic points of (firmly) nonexpansive mappings.

In the next section we study the asymptotic behaviour of Picard iterates of firmly nonexpansive mappings,
extending to $W$-hyperbolic spaces results of Reich and Shafrir \cite{ReiSha87,ReiSha90}.
As a consequence, we get that any firmly nonexpansive mapping with bounded orbits is asymptotically
regular.

A concept of weak convergence in geodesic spaces is the so-called $\Delta$-convergence, defined by Lim
\cite{Lim76}. Applying our asymptotic regularity result and general properties of Fej\' er monotone sequences,
we prove in Section \ref{cap:Picard-delta}, in the setting of complete $UCW$-hyperbolic spaces,
the $\Delta$-convergence of Picard iterates of a firmly nonexpansive mapping to a fixed point. As a consequence,
one gets the $\Delta$-convergence of a proximal point like algorithm to a minimizer of a proper, convex and
lower semicontinuous mapping defined on a CAT(0) space.

In the final section of the paper we obtain effective rates of asymptotic regularity for Picard iterations,
applying methods of proof mining, similar to the ones used for  Krasnoselski-Mann iterations of nonexpansive
mappings by Kohlenbach \cite{Koh03} in Banach spaces  and the second author \cite{Leu07} in $UCW$-hyperbolic
spaces. We point out that our results are new even for uniformly convex Banach spaces. In the case of
CAT(0) spaces we obtain a
quadratic rate of asymptotic regularity. {\em Proof mining} is a paradigm of research concerned with the
extraction, using tools from mathematical
logic, of hidden finitary and combinatorial content, such as algorithms and effective bounds, from proofs that
make use of highly infinitary principles. We refer to Kohlenbach's book \cite{Koh08-book} for details.

\section{Classes of geodesic spaces - definitions and properties}\label{cap:geodesic}

A {\em $W$-hyperbolic space}  $(X,d,W)$ is a metric space $(X,d)$ together with a convexity mapping $W:X\times X\times [0,1]\to X$ satisfying
\begin{eqnarray*}
(W1) & d(z,W(x,y,\lambda))\le (1-\lambda)d(z,x)+\lambda d(z,y),\\
(W2) & d(W(x,y,\lambda),W(x,y,\tilde{\lambda}))=|\lambda-\tilde{\lambda}|\cdot
d(x,y),\\
(W3) & W(x,y,\lambda)=W(y,x,1-\lambda),\\
(W4) & \,\,\,d(W(x,z,\lambda),W(y,w,\lambda)) \le (1-\lambda)d(x,y)+\lambda
d(z,w).
\end {eqnarray*}
The convexity mapping $W$ was first considered by Takahashi in \cite{Tak70}, where a triple
$(X,d,W)$ satisfying $(W1)$ is called a \defnterm{convex metric space}.
$W$-hyperbolic spaces  were introduced by Kohlenbach \cite{Koh05} and we refer to
\cite[p.384]{Koh08-book} for a comparison between them and other notions of 'hyperbolic space' that
can be found in the literature (see for example \cite{Kir82,GoeKir83,ReiSha90}).
The class of $W$-hyperbolic spaces includes (convex subsets of) normed spaces, the Hilbert ball
(see \cite{GoeRei84} for a book treatment) as well as $CAT(0)$ spaces \cite{BriHae99}.

We shall denote a $W$-hyperbolic space simply by $X$, when the metric $d$ and the mapping $W$ are
clear from the context. One can easily see that
\beq
d(x,W(x,y,\lambda))=\lambda d(x,y)\quad \text{~and~}\quad  d(y,W(x,y,\lambda))=(1-\lambda)d(x,y). \label{prop-xylambda}
\eeq
Furthermore, $W(x,y,0)=x,\,W(x,y,1)=y$ and $W(x,x,\lambda)=x$.

Let us recall now some notions concerning geodesics. Let $(X,d)$ be a metric space.
A \defnterm{geodesic path in} $X$ (\defnterm{geodesic in } $X$ for short) is a map $\gamma:[a,b]\to X$ satisfying
\beq
d(\gamma(s),\gamma(t))=|s-t| \text{~~for all~~} s,t\in [a,b].
\eeq
A \defnterm{geodesic segment} in $X$ is the image of a geodesic in $X$.
If $\gamma:[a,b]\to X$ is a geodesic in $X$, $\gamma(a)=x$ and $\gamma(b)=y$, we say that the
geodesic  $\gamma$ \defnterm{joins x and y} or that the geodesic  segment $\gamma([a,b])$
\defnterm{joins x and y}; $x$ and $y$ are also called the \defnterm{endpoints} of $\gamma$.

A metric space $(X,d)$ is said to be a \defnterm{(uniquely) geodesic space} if every two distinct
points are joined by a (unique) geodesic segment.

If $\gamma([a,b])$ is a geodesic segment joining $x$ and $y$ and $\lambda\in [0,1]$,
$z:=\gamma((1-\lambda)a+\lambda b)$ is the unique point in $\gamma([a,b])$ satisfying
\beq
d(z,x)=\lambda d(x,y)\quad \text{~and~}\quad d(z,y)=(1-\lambda)d(x,y).
\eeq
In the sequel, we shall use the notation $[x,y]$ for the geodesic segment $\gamma([a,b])$ and we shall denote this $z$  by
 $\lambdaxy$, provided that there is no possible ambiguity.

Given three points $x,y,z$ in a metric space $(X,d)$, we say that $y$ \defnterm{lies between} $x$ and
$z$ if these points are pairwise distinct and if we have $d(x,z)=d(x,y)+d(y,z)$.  Obviously,
if $y$ lies between $x$ and $z$, then $y$ also lies between $z$ and $x$. Furthermore, the relation  of betweenness satisfies also a transitivity property (see, e.g., \cite[Proposition 2.2.13]{Pap05}):

\bprop\label{trans-between}
Let $X$ be a metric space and $x,y,z,w$ be pairwise distinct points of $X$. The following statements are equivalent:
\be
\item $y$ lies between $x$ and $z$ and $z$ lies between $x$ and $w$.
\item $y$ lies between $x$ and $w$ and $z$ lies between $y$ and $w$.
\ee
\eprop

The following \defnterm{betweenness} property expresses another form of 'transitivity', which is not true
in general metric spaces: \beq
\bt{ll} for all $x,y,z,w\in X$, & if $y$ lies between $x$ and $z$ and $z$ lies between $y$ and $w$,\\
 & then $y$ and $z$ lie both between $x$ and $w$.
\et
 \label{bet-prop}
\eeq

By induction one gets

\blem\label{between-properties}
Let $X$ be a metric space satisfying (\ref{bet-prop}).
For all $n\ge 2$ and all $x_0,x_1,\ldots, x_n\in X$, we have that
\beq
\ba{l}
\text{ if for all }k=1,\ldots,n-1, \, x_k \text{ lies between }x_{k-1} \text{ and } x_{k+1}, \\
\text{ then for all }k=1,\ldots,n-1,
x_k \text{ lies between } x_0 \text{ and } x_{k+1}.
\ea
\label{bet-prop-2}
\eeq
\elem

The next lemma collects some well-known properties of geodesic spaces.
We refer to \cite{Pap05} for details.

\blem\label{geodesic-prop}
Let $(X,d)$ be a geodesic space.
\be
\item\label{geodesic-between} For every pairwise distinct points  $x,y,z$ in $X$,  $y$ lies
between $x$ and $z$ if and only if there exists a geodesic segment $[x,z]$ containing $y$.
\item\label{wz-xy-lies} For every points $x,y,z,w$ and any geodesic segment $[x,y]$,
if $z,w\in [x,y]$, then either $d(x,z)+d(z,w)=d(x,w)$ or $d(w,z)+d(z,y)=d(w,y)$.
\item\label{W2-geodesic} For every geodesic segment $[x,y]$ in $X$ and $\lambda,\tilde{\lambda}\in [0,1]$,
$$d\left(\lambdaxy, \tlambdaxy\right)=|\lambda - \tilde{\lambda}|d(x,y).$$
\item\label{gamma-yx} Let $\gamma:[a,b]\to X$ be a geodesic that joins $x$ and $y$. Define
\[\gamma^-:[a,b]\to X,\quad \gamma^-(s)=\gamma(a+b-s).\]
Then $\gamma^-$ is a geodesic that joins $y$ and $x$ such that $\gamma^-([a,b])=\gamma([a,b])$.
\item\label{uniq-geodesic} Let $\gamma, \eta:[a,b]\to X$ be geodesics. If $\gamma([a,b])=\eta([a,b])$ and $\gamma(a)=\eta(a)$ (or $\gamma(b)=\eta(b)$), then $\gamma=\eta$.
\item\label{uniq-geo-equiv} The following two statements are equivalent:
\be
\item $X$ is uniquely geodesic.
\item For any $x\ne y\in X$ and any $\lambda\in[0,1]$ there exists a unique element $z\in X$ such that
\[d(x,z)=\lambda d(x,y)\quad \text{~and~}\quad  d(y,z)=(1-\lambda)d(x,y).\]
\ee
\ee
\elem

\blem\label{uniq-geodesic-prop}
Let $X$ be a uniquely geodesic space.
\be
\item For all $x,y\in X$, $[x,y]=\{\lambdaxy\mid \lambda \in[0,1]\}$.
\item\label{between-geo-deg} For every pairwise distinct points  $x,y,z$ in $X$, $y$ lies between $x$ and
$z$ if and only if $y\in[x,z]$.
\item\label{useful-firmly-ne} Let $x,y,z,w$ be pairwise distinct points in $X$ such that
$y=\lambdaxz$ and $z=\alphaxw$ for some $\lambda,\alpha\in(0,1)$.
Then $z=\muyw$, where $\ds \mu=\frac{(1-\lambda)\alpha}{1-\alpha\lambda}$.
\ee
\elem
\begin{proof} (i),(ii) are obvious.\\
(iii) Applying (ii), Lemma \ref{geodesic-prop}.(\ref{geodesic-between}) and  Proposition \ref{trans-between},
one gets that $z\in[y,w]$. Thus, $z=\muyw$ for some $\mu\in(0,1)$. Furthermore,
\bua
d(z,y)&=& (1-\lambda)d(x,z)=(1-\lambda)\alpha d(x,w)=\frac{(1-\lambda)\alpha}{1-\alpha}d(z,w)\\
&=&\frac{(1-\lambda)\alpha}{1-\alpha}\cdot(1-\mu)d(y,w)=
\frac{(1-\lambda)\alpha}{1-\alpha}\cdot\frac{1-\mu}{\mu}d(z,y).
\eua
Thus, $\ds \frac{(1-\lambda)\alpha}{1-\alpha}\cdot\frac{1-\mu}{\mu}=1$ and the conclusion follows immediately.
\end{proof}

Let $(X,d,W)$ be a $W$-hyperbolic space. For all $x,y\in X$, let us define
\beq
[x,y]_W:=\{W(x,y,\lambda)\mid \lambda\in[0,1]\}. \label{def-xwW}
\eeq
Then $[x,x]_W=\{x\}$ for all $x\in X$. A subset $C\subseteq X$ is \defnterm{convex} if $[x,y]_W\se C$ for all $x,y\in C$.
Open and closed balls are convex sets. A nice feature of our setting is that any convex subset is itself a $W$-hyperbolic space.

Following \cite{Tak70}, we call a $W$-hyperbolic space \defnterm{strictly convex} if for any $x\ne y\in X$ and any $\lambda\in(0,1)$ there exists a unique element $z\in X$ (namely $z=W(x,y,\lambda)$) such that
\beq
d(x,z)=\lambda d(x,y)\quad \text{~and~}\quad  d(y,z)=(1-\lambda)d(x,y).\label{def-strict-convex-W-hyp}
\eeq

\bprop\label{prop-W-hyp}
Let $(X,d,W)$ be a $W$-hyperbolic space. Then
\be
\item\label{pWh-geodesic} $X$ is a geodesic space and for all $x\ne y\in X$, $[x,y]_W$ is a geodesic
segment joining $x$ and $y$.
\item\label{pWh-uniq-geodesic} $X$ is a uniquely geodesic space if and only if it is strictly convex.
\item If  $X$ is uniquely geodesic, then
\be
\item $W$ is the unique convexity mapping that makes $(X,d,W)$ a $W$-hyperbolic space.
\item For all $x,y\in X$ and $\lambda\in[0,1]$, $W(x,y,\lambda)=\lambdaxy$.
\ee
\ee
\eprop
\begin{proof}
\be
\item For $x\ne y\in X$, the map $W_{xy}:[0,d(x,y)]\to $,
\beq
W_{xy}(\alpha)=W\left(x,y,\frac{\alpha}{d(x,y)}\right).
\eeq\label{W-def-geodesic}
is a geodesic satisfying $W_{xy}([0,d(x,y)])=[x,w]_W$.
\item By Lemma \ref{geodesic-prop}.(\ref{uniq-geo-equiv}).
\item (b) is obvious. We prove in the sequel (a).  Let $W':X\times X\times[0,1]\to X$ be another convexity
mapping such that $(X,d,W')$ is
a $W$-hyperbolic space. For $\lambda\in [0,1]$ and $x\in X$ one has $W(x,x,\lambda)=W'(x,x,\lambda)=x$.
Let $x,y\in X, x\ne y$. Then $[x,y]_W$ and $[x,y]_{W'}$ are geodesic segments that join $x$ and $y$,
hence we must have that $[x,y]_W=[x,y]_{W'}$, that is $W_{xy}([0,d(x,y)])=W'_{xy}([0,d(x,y)])$.
Since $W_{xy}(0)=W'_{xy}(0)=x$, we can  apply Lemma \ref{geodesic-prop}.(\ref{uniq-geodesic}) to get
that $W_{xy}=W'_{xy}$, so that $W(x,y,\lambda)=W'(x,y,\lambda)$.
\ee
\end{proof}

An important class of $W$-hyperbolic spaces are the so-called Busemann spaces, used by
Busemann \cite{Bus48,Bus55} to define a notion of 'nonpositively curved space'.
We refer to \cite{Pap05} for an extensive study. Let us recall that a
map $\gamma:[a,b]\to X$ is an \defnterm{ affinely reparametrized geodesic} if
$\gamma$ is a constant path or there exist an interval $[c,d]$ and a geodesic
$\gamma':[c,d]\to X$ such that $\gamma=\gamma'\circ \psi$, where $\psi:[a,b]\to[c,d]$ is the unique
affine homeomorphism between the intervals $[a,b]$ and $[c,d]$.

A geodesic space $(X,d)$ is a \defnterm{Busemann space} if for any two affinely
reparametrized geodesics $\gamma:[a,b]\to X$ and $\gamma':[c,d]\to X$, the map
\beq
D_{\gamma,\gamma'}:[a,b]\times [c,d]\to \R, \quad D_{\gamma,\gamma'}(s,t)=d(\gamma(s),\gamma'(t))
\eeq
is convex.
Examples of Busemann spaces are strictly convex normed spaces. In fact, a normed space is a
Busemann space if and only if it is strictly convex.

\bprop\label{char-Busemann-W}
Let $(X,d)$ be a metric space. The following two statements are equivalent:
\be
\item $X$ is a Busemann space.
\item There exists a (unique) convexity mapping $W$ such that $(X,d,W)$ is a uniquely geodesic
$W$-hyperbolic space.
\ee
\eprop
\begin{proof} $(i)\Ra(ii)$ Assume that $X$ is Busemann. By \cite[Proposition 8.1.4]{Pap05}, any
Busemann space is uniquely geodesic. For any $x,y\in X$, let $[x,y]$ be the unique geodesic segment
that joins $x$ and $y$ and define
\beq
W:X\times X\times[0,1]\to X, \quad W(x,y,\lambda)=\lambdaxy.
\eeq
Let us verify (W1)-(W4): (W4) follows from \cite[Proposition 8.1.2.(ii)]{Pap05};
(W2) follows from Lemma \ref{geodesic-prop}.(\ref{W2-geodesic}); (W1) follows from (W4) applied with
$z=x$ and the fact that $W(x,x,\lambda)=x$; (W3) follows by
Lemma \ref{geodesic-prop}.(\ref{gamma-yx}).

$(ii)\Ra(i)$ Apply \cite[Proposition 8.1.2.(ii)]{Pap05} and (W4).
\end{proof}

A very useful feature of Busemann spaces is the following (see \cite[Proposition 8.2.4]{Pap05})

\blem\label{Bus-between}
Every Busemann space  satisfies the betweenness property (\ref{bet-prop}). Hence, Lemma \ref{between-properties} holds in Busemann spaces.
\elem

CAT(0) spaces are another very important class of $W$-hyperbolic spaces.  A {\em CAT(0) space}  is a geodesic
space satisfying the {\bf CN} inequality of Bruhat-Tits \cite{BruTit72}: for all $x,y,z\in X$  and all
$m\in X$ with $\ds d(x,m)=d(y,m)=\frac12 d(x,y)$,
\beq
d(z,m)^2\leq \frac12d(z,x)^2+\frac12d(z,y)^2-\frac14d(x,y)^2. \label{CN-ineq}
\eeq
We refer to \cite[p. 163]{BriHae99} for a proof that the above definition is equivalent with
the one using geodesic triangles. In the setting of $W$-hyperbolic spaces, we consider the following
reformulation of the {\bf CN} inequality: for all $x,y,z\in X$,
\bea
\mathbf{CN^-}: \quad\quad d\left(z,W\left(x,y,\frac12\right)\right)^2\leq \frac12d(z,x)^2+\frac12d(z,y)^2-\frac14d(x,y)^2. \label{CN-}
\eea

We refer to \cite[p. 386-388]{Koh08-book} for the proof of the following result.

\bprop\label{char-CAT0}
Let $(X,d)$ be a metric space. The following statements are equivalent:
\be
\item $X$ is a CAT(0) space.
\item There exists a (unique) convexity mapping $W$ such that $(X,d,W)$ is a $W$-hyperbolic space satisfying
the $\mathbf {CN^-}$ inequality (\ref{CN-}).
\ee
\eprop

{\bf Convention:} Given a W-hyperbolic space $(X,d,W)$ and $x,y\in X$, $\lambda\in[0,1]$, we shall use
from now on the notation $(1-\lambda)x\oplus \lambda y$ for $W(x,y,\lambda)$.

\subsection{UCW-hyperbolic spaces}

We define uniform convexity in the setting of $W$-hyperbolic spaces, following \cite[p. 105]{GoeRei84}. Thus, a
$W$-hyperbolic space $(X,d,W)$ is  \defnterm{uniformly convex} \cite{Leu07} if for
any $r>0$ and any $\varepsilon\in(0,2]$ there exists $\delta\in(0,1]$ such that
for all $a,x,y\in X$,
\begin{eqnarray}
\left.\begin{array}{l}
d(x,a)\le r\\
d(y,a)\le r\\
d(x,y)\ge\varepsilon r
\end{array}
\right\}
& \quad \Rightarrow & \quad d\left(\frac12x\oplus\frac12y,a\right)\le (1-\delta)r. \label{Ishikawa-uc-def}
\end{eqnarray}
A mapping $\eta:(0,\infty)\times(0,2]\rightarrow (0,1]$ providing such a
$\delta:=\eta(r,\varepsilon)$ for given $r>0$ and $\varepsilon\in(0,2]$ is called a
\defnterm{ modulus of uniform convexity}. We call $\eta$ \defnterm{monotone} if it decreases with
$r$ (for a fixed $\eps$).

\bprop\label{UCW-Busemann}
Any uniformly convex $W$-hyperbolic space is a Busemann space.
\eprop
\begin{proof}
Apply \cite[Proposition 5]{Leu07}, Proposition \ref{prop-W-hyp}.(\ref{pWh-uniq-geodesic}) and Proposition
\ref{char-Busemann-W}.
\end{proof}

Following \cite{Leu10}, we shall refer to  uniformly convex $W$-hyperbolic spaces with a monotone modulus of
uniform convexity as \defnterm{$UCW$-hyperbolic spaces}. Furthermore, we shall also use the notation
$(X,d,W,\eta)$ for a
$UCW$-hyperbolic space having $\eta$ as a monotone modulus of uniform convexity.

As it was proved in \cite{Leu07}, $CAT(0)$ spaces are  $UCW$-hyperbolic spaces with a modulus of uniform
convexity $\ds\eta(r,\varepsilon)=\frac{\varepsilon^2}{8}$, that does not depend on $r$ and is quadratic in
$\eps$. In particular, any $CAT(0)$ space is also a Busemann space.

The following lemma collects some useful properties of $UCW$-hyperbolic spaces. We refer to \cite{Leu07,Leu10} for the proofs.

\blem\label{eta-prop-1}
Let $(X,d,W,\eta)$ be a $UCW$-hyperbolic space. Assume that $r>0$, $\varepsilon\in(0,2]$ and $a,x,y\in X$ are such that
\[d(x,a)\le r,\,\,d(y,a)\le r \text{~and~} d(x,y)\ge\eps r.\]
Let $\lambda\in[0,1]$ be arbitrary.
\be
\item\label{Groetsch-eta} $\ds d(\lambdaxy,a)\le  \big(1-2\lambda(1-\lambda)\eta(r,\varepsilon)\big)r$;
\item\label{eta-monotone-eps} For any  $0<\psi\le\eps$,
\[\ds d(\lambdaxy,a)\le  \big(1-2\lambda(1-\lambda)\eta(r,\psi)\big)r\,;\]
\item\label{eta-monotone-s-geq-r} For any $s\geq r$,
\[d(\lambdaxy,a) \le \left(1-2\lambda(1-\lambda)\eta\left(s,\eps\right)\right)r\,.\]
\ee
\elem

\subsection{Asymptotic centers}

One of the most useful tools in metric fixed point theory is the asymptotic center technique, introduced by Edelstein \cite{Ede72,Ede74}.

Let $(X,d)$ be a metric space, $(x_n)$ be a bounded sequence in $X$ and $C\se X$ be a nonempty subset of $X$. We define the following functional:
\bua
r(\cdot, (x_n)):X\to[0,\infty), \quad r(y,(x_n)) &=&\lsupn d(y,x_n).
\eua
The {\em asymptotic radius} of $(x_n)$ {\em with respect to} $C$ is given by
\[r(C,(x_n))  = \inf\{r(y,(x_n))\mid y\in C\}.\]

A point $c\in C$ is said to be an \defnterm{asymptotic center} of $(x_n)$ {\em with respect to } $C$ if
\[r(c,(x_n))= r(C,(x_n))=\min\{r(y,(x_n))\mid y\in C\}.\]
We denote with $A(C,(x_n))$ the set of asymptotic centers of $(x_n)$ with respect to $C$. When $C=X$, we call $c$ an \defnterm{asymptotic center} of $(x_n)$ and we use the notation $A((x_n))$ for $A(X,(x_n))$.

The following lemma will be very useful in the sequel.

\blem\label{UCW-useful-unique-as-center} \cite{Leu10}\\
Let $(x_n)$ be a bounded sequence in $X$ with $A(C,(x_n))=\{c\}$ and  $(\alpha_n),(\beta_n)$ be real sequences such that $\alpha_n\geq 0$ for all $n\in\N$,
$\lsupn \alpha_n\leq 1$ and $\lsupn \beta_n\leq 0$.\\
 Assume that $y\in C$  is such that there exist $p,N\in\N$ satisfying
\[\forall n\ge N\bigg(d(y,x_{n+p})\leq \alpha_nd(c,x_n)+\beta_n\bigg).\]
Then $y=c$.
\elem

A classical result is the fact that in uniformly convex Banach spaces, bounded sequences have unique asymptotic centers with
respect to closed convex subsets. For the Hilbert ball, this was proved in \cite[Proposition 21.1]{GoeRei84}.
The following result shows that the same is true for complete $UCW$-hyperbolic spaces.

\bprop\label{UCW-unique-ac}\cite{Leu10}
Let $(X,d,W)$ be a complete $UCW$-hyperbolic space. Every bounded sequence $(x_n)$ in $X$ has a unique asymptotic center with
respect to any nonempty closed convex subset $C$ of $X$.
\eprop

\subsection{Convex functions}

Let $(X,d)$ be a geodesic space and $F:X\to(-\infty,\infty]$. The mapping $F$ is said to be {\em convex} if, for any
geodesic $\gamma$ in $X$, the function $F\circ \gamma$ is convex. Let us recall that the {\em effective domain}
of $F$ is the set $dom\, F:=\{x\in X \mid F(x)<\infty\}$ and that $F$ is {\em proper} if $dom\, F$ is nonempty.

For the rest of this section $F$ is a proper convex function.

If $x\in dom\, F$ and $\gamma:[0,c]\to X$ is a geodesic starting at $x$, the {\em directional derivative} $D_\gamma F(x)$ of $F$ at $x$ in the direction $\gamma$ is defined by
$$D_\gamma F(x):=\lim_{t\to 0^+}\frac{F(\gamma(t))-F(x)}{t}.$$
As $F$ is convex, the above limit (possibly infinite) always exist. Indeed, one can easily see that $\ds D_\gamma F(x)=\inf_{t>0}\frac{F(\gamma(t))-F(x)}{t}$.

\begin{proposition}\label{local-global-dir-der}
Let $\ol{x}\in dom\, F$. The following statements are equivalent:
\be
\item $\ol{x}$ is a local minimum of $F$.
\item $\ol{x}$ is a global minimum of $F$.
\item $D_\gamma F(\ol{x})\geq 0$ for any geodesic $\gamma:[0,c]\to X$ starting at $\ol{x}$.
\ee
\end{proposition}
\begin{proof}
$(i)\Ra(ii)$  Let $\eps>0$ be such that  $F(\ol{x})\leq F(x)$ for all $x\in B(\ol{x},\eps)$. Let $z\ne \ol{x}$ be arbitrary
and $\gamma:[0,d(\ol{x},z)]\to X$ be a geodesic in $X$ that joins $\ol{x}$ and $z$.  For all $t< \min\{\eps,d(\ol{x},z)\}$
we have that $d(\gamma(t),\ol{x})=t< \eps$, so that
\bua
F(\ol{x})& \leq & F(\gamma(t))=(F\circ\gamma)\left(\left(1-\frac{t}{d(\ol{x},z)}\right)0+\frac{t}{d(\ol{x},z)}d(\ol{x},z)\right)\\
& \leq & \left(1-\frac{t}{d(\ol{x},z)}\right)F(\ol{x})+\frac{t}{d(\ol{x},z)}F(z).
\eua
Hence $F(\ol{x})\leq F(z)$.

$(ii)\Ra(i)$ and $(ii)\Ra(iii)$  are immediate.

$(iii)\Ra(ii)$ Let $z\ne \ol{x}$ be arbitrary. We shall prove that $F(z)\geq F(\ol{x})$.
If $F(z)=\infty$, the conclusion is obvious, so we can assume that $z\in dom\, F$.
Let $\gamma:[0,d(\ol{x},z)]\to X$ be a geodesic that joins $\ol{x}$ and $z$. As $F$ is convex, one gets that for all
$t\in [0,d(\ol{x},z)]$,
\bua
F(\gamma(t))&\leq &\left(1-\frac{t}{d(\ol{x},z)}\right)F(\ol{x})+\frac{t}{d(\ol{x},z)}F(z)<\infty.
\eua
Thus, $F\comp\gamma:[0,d(\ol{x},z)]\to\R$ is a convex real function satisfying
$(F\comp\gamma)'(0)=D_\gamma F(\ol{x})\geq 0$. One gets that
$F(\ol{x})=(F\comp\gamma)(0)\leq(F\comp\gamma)(d(\ol{x},z))=F(z)$.
\end{proof}

\section{Firmly nonexpansive mappings}\label{cap:firmly-ne}

Firmly nonexpansive mappings were introduced by Bruck \cite{Bru73} in the context of Banach spaces and
by Browder \cite{Bro67}, under the name of {\em firmly contractive}, in the setting of Hilbert spaces. We refer to
\cite[Section 24]{GoeRei84} for a study of this class of mappings in the Hilbert ball.

Bruck's definition can be extended to $W$-hyperbolic spaces. Let $(X,d,W)$ be a $W$-hyperbolic space, $C\se X$ and $T:C\to X$.
Given $\lambda\in (0,1)$, we say that
$T$ is \defnterm{$\lambda$-firmly nonexpansive} if for all $x,y\in C$,
\beq
d(Tx,Ty)\leq d((1-\lambda)x\oplus\lambda Tx,(1-\lambda)y\oplus \lambda Ty)\quad \text{for all }x,y\in C.
\label{def-lambda-fe}
\eeq
If (\ref{def-lambda-fe}) holds for all $\lambda\in (0,1)$, then $T$ is said to be {\em firmly nonexpansive}.

Applying (W4) one gets that any $\lambda$-firmly nonexpansive mapping is nonexpansive, i.e. it satisfies
$d(Tx,Ty)\le d(x,y)$ for all $x,y\in C$.

The first example of a firmly nonexpansive mapping is the metric projection in a CAT(0) space. Let us
recall that a subset $C$ of a metric space $(X,d)$ is called a {\em Chebyshev set} if to each point
$x\in X$ there corresponds a unique point $z\in C$ such that
$d(x,z)= d(x,C)$, where $d(x,C)=\inf\{d(x,y)\mid y\in C\}$. If $C$ is a Chebyshev set, the
{\em metric projection}
$P_C:X\to C$ can be defined by assigning $z$ to $x$.

By \cite[Proposition II.2.4]{BriHae99}, any closed convex subset $C$ of a CAT(0) space is a Chebyshev
set, the metric projection $P_C$ is nonexpansive and $P_C((1-\lambda)x\oplus\lambda P_Cx)=P_C(x)$ for
all $x\in X$ and all $\lambda\in (0,1)$. It is well known that in the setting of Hilbert spaces the metric projection
is firmly nonexpansive. We remark that for the Hilbert ball this was proved in \cite[p. 111]{GoeRei84}.
The following result shows
that the same holds in general CAT(0) spaces.

\bprop
Let $C$ be a nonempty closed convex subset of a CAT(0) space $(X,d)$. The metric projection $P_C$ onto
$C$ is a firmly nonexpansive mapping.
\eprop
\begin{proof}
Let $x,y\in X$ and $\lambda\in(0,1)$. One gets that
\bua
d(P_Cx,P_Cy)&= & d\big(P_C((1-\lambda)x\oplus\lambda P_Cx),P_C((1-\lambda)y\oplus\lambda P_Cy)\big) \\
                &\leq & d((1-\lambda)x\oplus\lambda P_Cx,(1-\lambda)y\oplus\lambda P_Cy).
\eua
\end{proof}

Bruck \cite{Bru73} showed for Banach spaces that one can associate to any nonexpansive mapping a
family of firmly nonexpansive mappings having the same fixed points. Goebel and Reich \cite{GoeRei84}
obtained the same result for the Hilbert ball. We show in the sequel that Bruck's construction can be
adapted also to Busemann spaces.

Let $C$ be a nonempty closed convex subset of a complete Busemann space $X$ and $T:C\to C$ be nonexpansive. For $t\in (0,1)$ and $x\in C$ define
\beq
T_t^x:C\to C, \quad T_t^x(y)=(1-t)x\oplus tT(y). \label{averaged-u}
\eeq
Using (W4), one can easily see that $T_t^x$ is a contraction, so it has a unique fixed point $z_t^x\in C$, by
Banach's Contraction Mapping Principle. Let
\beq
U_t:C\to C, \quad U_t(x)=z_t^x.
\eeq
Then $U_t(x)=(1-t)x\oplus tT(U_t(x))$ for all $x\in C$.

\bprop
$U_t$ is a firmly nonexpansive mapping having the same set of fixed points as $T$.
\eprop
\begin{proof}
Let $\lambda\in(0,1)$ and $x,y\in C$. Denote $u:=(1-\lambda)x\oplus \lambda U_t(x)$ and
$v:=(1-\lambda)y\oplus \lambda U_t(y)$.  We can apply Lemma \ref{uniq-geodesic-prop}.(\ref{useful-firmly-ne}) twice to get that
\bua
U_t(x)=(1-\mu)u\oplus \mu T(U_t(x)), \,\,\,\, U_t(y)=(1-\mu)v\oplus \mu T(U_t(y)), \quad{ where }\,
\mu=\frac{(1-\lambda)t}{1-\lambda t}.
\eua
It follows that
\bua
d(U_t(x),U_t(y)) &=& d((1-\mu)u\oplus \mu T(U_t(x)), (1-\mu)v\oplus \mu T(U_t(y)))\\
&\leq & (1-\mu)d(u,v)+\mu d(T(U_t(x)),T(U_t(y)))\\
&\leq & (1-\mu)d(u,v)+\mu d(U_t(x),U_t(y)).
\eua
Thus, $d(U_t(x),U_t(y)) \leq d(u,v)$, so $U_t$ is $\lambda$-firmly nonexpansive.

The fact that $U_t$ and $T$ have the same set of fixed points is immediate.

\end{proof}

A third example of a firmly nonexpansive mapping is the resolvent of a proper, convex and lower semicontinuous
mapping in a CAT(0) space.

Let $(X,d)$ be a CAT(0) space, $F:X\to(-\infty,\infty]$ and $\mu>0$. Following Jost \cite{Jos95}, the
{\em Moreau-Yosida approximation} $F^\mu$ of $F$ is defined by
\begin{equation}\label{eq:def-F-mu}
 F^\mu(x):=\inf_{y\in X}\big\{\mu\,F(y)+d(x,y)^2\big\}.
\end{equation}
We refer to \cite{Bac12,Sto11} for applications of the Moreau-Yosida approximation in CAT(0) spaces.

Jost proved \cite[Lemma 2]{Jos95} that if $F:X\to(-\infty,\infty]$ is proper, convex and lower semicontinuous,
then  for every $x\in X$ and $\mu>0$, there exists a unique $y_\mu\in X$ such that
 $$F^\mu(x) = \mu\,F(y_\mu )+d(x,y_\mu)^2.$$
We denote this $y_\mu$ with $J_\mu(x)$ and call $J_\mu$  the {\em  resolvent} of $F$ of order $\mu$.

In the same paper, Jost shows that for all $\mu >0$ the resolvent $J_\mu$ is nonexpansive
\cite[Lemma 4]{Jos95} and, furthermore, that for all $\lambda\in[0,1]$,
\begin{equation}\label{eq:J-mu-x-s}
J_{(1-\lambda)\mu}\big((1-\lambda)x\oplus \lambda J_\mu(x)\big)=J_\mu(x) \quad (\text{see
\cite[Corollary 1]{Jos95}})
\end{equation}

\bprop\label{resolvent-firmly-ne}
Let $F:X\to(-\infty,\infty]$ be proper, convex and lower semicontinuous. Then for every $\mu>0$, its resolvent
$J_\mu$ is a firmly nonexpansive mapping.
\eprop
\begin{proof}
Let $x,y\in X$ and $\lambda\in(0,1)$. Then
\bua
d(J_\mu(x),J_\mu(y))&=& d(J_{(1-\lambda)\mu}\big((1-\lambda)x\oplus \lambda J_\mu(x)\big),J_{(1-\lambda)\mu}
\big((1-\lambda)y\oplus \lambda J_\mu(y)\big))\\
&\leq & d\big((1-\lambda)x\oplus \lambda J_\mu(x),(1-\lambda)y\oplus \lambda J_\mu(y)\big).
\eua
\end{proof}

Another example of a firmly nonexpansive mapping, given by Kopeck\'{a} and Reich \cite[Lemma 2.2]{KopRei09}, is the
resolvent of a coaccretive operator in the Hilbert ball.

\section{A fixed point theorem}\label{cap:fpt-theorem}

Given a subset $C$ of a metric space $(X,d)$, a nonexpansive mapping
$T:C\to C$ and $x\in C$, the {\em orbit} $\cO(x)$ of $x$ under $T$ is defined by
$\cO(x)=\{T^nx\mid n=0,1,2,\ldots\}$.  As an immediate consequence of the nonexpansiveness of $T$, if $\cO(x)$ is bounded for some $x\in C$, then all other orbits $\cO(y), \, y\in C$ are bounded. If this is the case, we say that $T$ {\em has bounded orbits}. Obviously, if $T$ has fixed points, then $T$ has bounded orbits.

In this section we prove the following fixed point theorem.

\bthm\label{thm-fpt-gen-Sma}
Let $(X,d,W)$ be a complete $UCW$-hyperbolic space, $\ds C=\bigcup_{k=1}^pC_k$ be a union of nonempty closed
convex subsets $C_k$ of $X$, and $T:C\to C$ be $\lambda$-firmly nonexpansive for some $\lambda\in(0,1)$.
The following two statements are equivalent:
\be
\item $T$ has bounded orbits.
\item $T$ has fixed points.
\ee
\ethm

Let us remark that fixed points are not guaranteed if $T$ is merely nonexpansive, as the following trivial example shows. Let $x\ne y\in X$, take $C_1=\{x\},\, C_2=\{y\}, \, C=C_1\cup C_2$ and $T:C\to C, \, T(x)=y, T(y)=x$. Then $T$ is fixed point free and nonexpansive. If $T$ were $\lambda$-firmly nonexpansive for some $\lambda\in(0,1)$, we would get
\bua
0<d(x,y)&=& d(Tx,Ty)\leq d((1-\lambda)x\oplus\lambda Tx,(1-\lambda)y\oplus \lambda Ty)\\
&=& d((1-\lambda)x\oplus\lambda y,\lambda x\oplus (1-\lambda)y)=|2\lambda-1|d(x,y)\quad \text{by (W2)}\\
&< & d(x,y),
\eua
that is a contradiction.

As an immediate consequence, we get a strengthening of Smarzewski's fixed point theorem for uniformly convex
Banach spaces \cite{Sma91}, obtained by weakening the hypothesis of $C_k$ being bounded for all $k=1,\ldots, p$
to $T$ having bounded orbits.

\bcor\label{fpt-cor-firmly}
Let $X$ be a uniformly convex Banach space, $\ds C=\bigcup_{k=1}^pC_k$ be a union of nonempty closed convex
subsets $C_k$ of $X$, and $T:C\to C$ be $\lambda$-firmly nonexpansive for some $\lambda\in(0,1)$.

Then $T$ has fixed points if and only if  $T$ has bounded orbits.
\ecor

Theorem \ref{thm-fpt-gen-Sma} follows from the following Propositions \ref{UCW-T-ne-periodic} and
\ref{firmly-ne-periodic-fixed}.

\bprop\label{firmly-ne-periodic-fixed}
Let $X$ be a Busemann space, $C\se X$ be nonempty and $T:C\to C$ be $\lambda$-firmly nonexpansive for some
$\lambda\in(0,1)$. Then any periodic point of $T$ is a fixed point of $T$.
\eprop
\begin{proof}
Let $x$ be a periodic point of $T$  and $m\ge 0$ be minimal with the property that $T^{m+1}x=x$. If $m=0$,
then $x$ is a fixed point of $T$, hence we can assume that $m\ge 1$.
Since $T$ is nonexpansive, we have
\bua
d(x,T^mx)&=& d(T^{m+1}x,T^mx)\leq d(T^mx,T^{m-1}x)\leq\cdots\leq d(Tx,x)\\
&=& d(Tx,T^{m+1}x)\leq d(x,T^mx),
\eua
hence we must have equality everywhere, that is
\beq
d(Tx,x)=d(T^2x,Tx)=\ldots =d(T^mx,T^{m-1}x)=d(x,T^mx):=\gamma>0,
\eeq
since $T^mx\ne x$, by the hypothesis on $m$.
Applying now the fact that $T$ is $\lambda$-firmly nonexpansive, we get for all $k=1,\ldots,m$
\bua
\gamma&= & d(T^{k+1}x,T^kx)\le d((1-\lambda)T^kx\oplus\lambda T^{k+1}x, (1-\lambda)T^{k-1}x\oplus\lambda T^kx)\\
&\le & d((1-\lambda)T^kx\oplus\lambda T^{k+1}x,T^kx)+d(T^kx, (1-\lambda)T^{k-1}x\oplus\lambda T^kx)\\
&=& \lambda d(T^kx,T^{k+1}x)+(1-\lambda)d(T^{k-1}x,T^kx)=\lambda \gamma+(1-\lambda)\gamma=\gamma.
\eua
Hence, we must have
\beq
\gamma= d(\alpha_k, \beta_k)=d(\alpha_k, T^kx)+d(T^kx,\beta_k),
\eeq
where
\bua
\alpha_k &:= & (1-\lambda)T^kx\oplus\lambda T^{k+1}x,\\
\beta_k&:= &(1-\lambda)T^{k-1}x\oplus\lambda T^kx=\lambda T^kx\oplus (1-\lambda)T^{k-1}x.
\eua
We have the following cases:
\be
\item $m=1$, hence $k=1$. Then $T^{m-1}x=x$ and
\[\alpha_1:=(1-\lambda)Tx\oplus\lambda x,\quad \beta_1=\lambda Tx\oplus (1-\lambda)x.\]
 It follows by (W2) that
\bua
\gamma=d(\alpha_1, \beta_1) = |\lambda-(1-\lambda)|d(x,Tx)=|2\lambda-1|\gamma,
\eua
hence $|2\lambda-1|=1$, which is impossible, since $\lambda\in (0,1)$.
\item $m\ge 2$, hence  $m-1\ge 1$. Since $T^kx$ lies between $\beta_k$ and $\alpha_k$ and, furthermore,
$\alpha_k$ lies between $T^kx$ and $T^{k+1}x$, we can apply Lemma \ref{Bus-between} twice to get firstly that
$T^kx$ lies between $\beta_k$ and $T^{k+1}x$ and secondly, since $\beta_k$ lies between $T^{k-1}x$ and $T^kx$,
that $T^kx$ lies between $T^{k-1}x$ and $T^{k+1}x$ for all $k=1,\ldots, m$.

Apply now Lemma \ref{between-properties} to conclude that $T^{m-1}x$ lies between $x$ and $T^mx$, hence
\bua
\gamma = d(x,T^mx)= d(T^{m-1}x,T^mx)+d(T^{m-1}x,x)=\gamma+d(T^{m-1}x,x)>\gamma,
\eua
since $T^{m-1}x\ne x$. We have got a contradiction.
\ee
\end{proof}

We remark that Proposition \ref{firmly-ne-periodic-fixed} holds for strictly convex Banach spaces too, as
they are Busemann spaces.

\blem\label{Smarzewski-lemma-1}
Let $(X,d)$ be a metric space, $\ds C=\bigcup_{k=1}^pC_k$ be a union of nonempty subsets $C_k$ of $X$, and
$T:C\to C$ be nonexpansive. Assume that $T$ has bounded orbits and that for some $z\in C$, the orbit $(T^nz)$
of $T$ has a unique asymptotic center $x_k$ with respect to every $C_k,\, k=1,\ldots, p$.

Then one of $x_k,\, k=1,\ldots, p$ is a periodic point of $T$.
\elem
\begin{proof}
Since $T$ is nonexpansive, we have that
\bea
d(Tx_k,T^{n+1}z)&\leq& d(x_k,T^nz) \quad\text{for all }n\ge 0, \,k=1,\ldots,p, \text{ hence}
\label{smarzewski-eq-1}\\
r(Tx_k,(T^nz))& \leq & r(x_k,(T^nz)) \quad\text{for all } \,k=1,\ldots,p.\label{smarzewski-eq-2}
\eea
If there exists $k_0\in\{1,\ldots,p\}$ such that $Tx_{k_0}\in C_{k_0}$, then applying
Lemma \ref{UCW-useful-unique-as-center} with $y=Tx_{k_0}, p=1, \alpha_n=1, \beta_n=0$ and $x_n=T^nz$, we have
that $Tx_{k_0}=x_{k_0}$, that is, $x_{k_0}$ is a fixed point of $T$. In particular, $x_{k_0}$ is a periodic
point of $T$.

Otherwise, assume that $Tx_k\not\in C_k$ for all $1\leq k\leq p$. It is easy to see that there exist integers
$\{n_1,n_2,\ldots,n_m\}\subseteq\{1,2,\ldots,p\}$, with $m\geq2$, such that $Tx_{n_k}\in C_{n_{k+1}}$ for all
$k=1,\ldots,m-1$ and $Tx_{n_m}\in C_{n_1}$.

Applying repeatedly (\ref{smarzewski-eq-2}) and the fact that $x_{n_k}$ is the unique asymptotic center of
$(T^nz)$ with respect to $C_{n_k}$, we get that
\bua
 r(x_{n_1},(T^nz))& \le & r(Tx_{n_m},(T^nz)) \leq r( x_{n_m},(T^nz)\}) \leq \ldots \leq r(Tx_{n_1},(T^nz))\\
 &\leq & r(x_{n_1},(T^nz)).
\eua
Thus, we must have equality everywhere. We get that $r(x_{n_1},(T^nz))=r(Tx_{n_m},(T^nz))$ and
$r(Tx_{n_k},(T^nz))=r(x_{n_{k+1}},(T^nz))$ for all $k=1,\ldots,m-1$. By the uniqueness of the asymptotic
centers, we get that
\beq
x_{n_1}=Tx_{n_m}\text{ and } Tx_{n_k}=x_{n_{k+1}}\text{ for all } k=1,\ldots,m-1.
\eeq
It follows that $T^{m}x_{n_1}=x_{n_1}$, hence $x_{n_1}$ is a periodic point of $T$.
\end{proof}

\bprop\label{UCW-T-ne-periodic}
Let $(X,d,W)$ be a complete $UCW$-hyperbolic space, $\ds C=\bigcup_{k=1}^pC_k$ be a union of nonempty closed
convex subsets $C_k$ of $X$, and $T:C\to C$ be a nonexpansive mapping having bounded orbits.

Then $T$ has periodic points.
\eprop
\begin{proof}
By Proposition \ref{UCW-unique-ac}, for all $z\in C$ and for all $k=1,\ldots, p$, the orbit $(T^nz)$ has a
unique asymptotic center $x_k$ with respect to $C_k$. Apply Lemma \ref{Smarzewski-lemma-1} to get that one of
the asymptotic centers $x_k,\, k=1,\ldots,p$ is a periodic point of $T$.
\end{proof}

\section{Asymptotic behaviour of Picard iterations}\label{cap:Picard-iterates}

The second main result of the paper is a theorem on the asymptotic behaviour of Picard iterations of
$\lambda$-firmly nonexpansive mappings, which generalizes results obtained by Reich and Shafrir \cite{ReiSha87}
for firmly nonexpansive mappings in Banach spaces and the Hilbert ball.

\bthm\label{Picard-thm}
Let $C$  be a subset of a W-hyperbolic space $X$ and $T:C\to C$  be a $\lambda$-firmly nonexpansive
mapping with $\lambda\in(0,1)$. Then for all $x\in X$ and $k\in\Z_+$,
\bua
\limn d(T^{n+1}x,T^nx)=\frac{1}{k}\limn d(T^{n+k}x,T^nx)=\limn\frac{d(T^nx,x)}{n}=r_C(T),
\eua
where $r_C(T):=\inf\{d(x,Tx)\mid x\in C\}$ is the {\em minimal displacement } of  $T$.
\ethm

The mapping $T$ is said to be {\em asymptotically regular at} $x\in C$ if
$\displaystyle\limn d(T^nx,T^{n+1}x)=0$. If this is true for all $x\in C$, we say that $T$ is
{\em asymptotically regular}.

Before proving Theorem \ref{Picard-thm}, we give the following immediate consequences.

\bcor\label{rCT-as-reg}
The following statements are equivalent:
\be
\item $T$ is asymptotically regular at some $x\in C$.
\item $r_C(T)=0$.
\item $T$ is asymptotically regular.
\ee
\ecor

\bcor\label{T-bounded-orbits-as-reg}
If $T$ has bounded orbits, then $T$ is asymptotically regular.
\ecor

\begin{remark}
As Adriana Nicolae pointed out to us in a private communication, one can easily see that Proposition \ref{firmly-ne-periodic-fixed}
is an immediate consequence of the above corollary.  However, our proof of this proposition holds (with small adaptations) also in
more general spaces like geodesic spaces with the betweenness property (see \cite{Nic12}), for which it is not known whether
Corollary  \ref{T-bounded-orbits-as-reg} is true.
\end{remark}

\subsection{Proof of Theorem \ref{Picard-thm}}\label{cap:proof-Picard-iterations}

In the sequel, $X$ is a $W$-hyperbolic space, $C\se X$ and $T:C\to C$.

\blem\label{lema1-Picard-thm}
Assume that $T$ is nonexpansive and $x\in C$.
\be
\item\label{Rk-R1} For all $k\ge 1$, $\ds R_k:=\limn d(T^{n+k}x,T^nx)$ exists and $R_k\le kR_1$.
\item $L:=\ds \limn\frac{d(T^nx,x)}{n}$ exists and equals $\ds \inf_{n\ge 1}\frac{d(T^nx,x)}{n}$. Moreover,
$L$ is independent of  $x$.
\item\label{L-le-R1} $L\le r_C(T)\le R_1$.
\ee
\elem
\begin{proof}
\be
\item Since the sequence $(d(T^{n+k}x,T^nx))_n$ is nonincreasing, obviously $R_k$ exists. Remark that
\bua
d(T^{n+k}x,T^nx)\leq \sum_{i=0}^{k-1}d(T^{n+i}x,T^{n+i+1}x)\le kd(T^{n+1}x,T^nx)
\eua
and let $n\to\infty$ to conclude that $R_k\le kR_1$.
\item One has that for all $m,n\ge 1$,
\bua
d(T^{m+n}x,x)\leq d(T^{m+n}x,T^nx)+d(T^nx,x)\le d(T^mx,x)+d(T^nx,x),
\eua
hence the sequence $(d(T^nx,x))$ is subadditive. Apply now Fekete's subadittive lemma \cite{Fek23} to get
that $\ds L=\inf_{n\ge 1}\frac{d(T^nx,x)}{n}$. The independence of $x$ follows from the fact that for all
$x,y\in C$,
\bua
d(T^nx,x)-d(T^ny,y) \le d(T^nx,T^ny)+d(x,y)\le 2d(x,y).
\eua
\item Obviously, $\ds R_1=\inf_{n\ge 1} d(T^nx,T^{n+1}x)\ge r_C(T)$. Given  $\eps>0$, there is a point $y\in C$ such that $r_C(T)\leq d(y,Ty)<r_C(T)+\eps$. It follows that
\bua
L & = & \limn \frac{d(T^nx,x)}{n}=\limn \frac{d(T^ny,y)}{n}=\inf_{n\ge 1}\frac{d(T^ny,y)}{n}\le d(Ty,y)\\
& <  & r_C(T)+\eps.
\eua
As $\eps>0$ was arbitrary, we get that $L\le r_C(T)$.
\ee
\end{proof}

\blem
Let $T$ be $\lambda$-firmly nonexpansive for some $\lambda\in(0,1)$. Then for all $x,y\in C$,
\beq
d(Tx,Ty)\leq \frac{1-\lambda}{1+\lambda}d(x,y)+\frac{\lambda}{1+\lambda}(d(Tx,y)+d(x,Ty)).
\label{ineq-lambda-fe-1}
\eeq
\elem
\begin{proof}
Apply (W1) more times to get that
\bua
d(Tx,Ty)&\leq & d((1-\lambda)x\oplus\lambda Tx,(1-\lambda)y\oplus \lambda Ty))\\
&\leq & (1-\lambda)d((1-\lambda)x\oplus\lambda Tx,y)+\lambda d((1-\lambda)x\oplus\lambda Tx,Ty)\\
&\leq & (1-\lambda)^2d(x,y)+\lambda(1-\lambda)\big(d(Tx,y)+d(x,Ty)\big)+\lambda^2 d(Tx,Ty).
\eua
\end{proof}

{\bf Proof of Theorem \ref{Picard-thm}}\\

\noindent We prove that $R_k=kR_1$ for all $k\ge 1$ by induction on $k$. Assume that $R_j=jR_1$
for all $j=1,\ldots, k$ and let $\eps >0$. Since $(d(T^{n+j}x,T^nx))$ is nonincreasing, we get $N_\eps\ge 1$ such that
for any $j=1,\ldots, k$ and for all $n\geq N_\eps$,
\beq
R_1\leq\frac{1}{j}\,d(T^{n+j}x,T^nx)\leq R_1+\eps
\eeq
Let $n\ge N_\eps$. By (\ref{ineq-lambda-fe-1}), we get that
$$\ds d(T^{n+1}x,T^{n+k+1}x) \leq  \frac{1-\lambda}{1+\lambda}d(T^nx,T^{n+k}x)
+\frac{\lambda}{1+\lambda}\big(d(T^{n+1}x,T^{n+k}x)+d(T^nx,T^{n+k+1}x)\big),$$
hence
\bua
d(T^nx,T^{n+k+1}x)& \geq &  \frac{1+\lambda}{\lambda}d(T^{n+1}x,T^{n+k+1}x)-\frac{1-\lambda}{\lambda}d(T^{n+k}x,T^nx)-\\
&& - d(T^{n+1}x,T^{n+k}x)\\
&\geq & \frac{1+\lambda}{\lambda}\,k\,R_1-\frac{1-\lambda}{\lambda}\,k\,(R_1+\eps)-(k-1)(R_1+\eps)\\
& = & (k+1)\,R_1+\left(1-\frac{k}{\lambda}\right)\,\eps.
\eua
By letting $n\to\infty$, it follows that $R_{k+1}\ge (k+1)R_1$, as $\eps>0$ is arbitrary. Apply now
Lemma \ref{lema1-Picard-thm}.(\ref{Rk-R1}) to conclude that  $R_{k+1}=(k+1)R_1$.

Since $d(T^{n+k}x,T^nx)\leq d(T^kx,x)$ for all $k,n\ge 1$, let $n\to\infty$ to get that $\ds R_1\leq \frac{d(T^kx,x)}k$ for all $k\ge 1$ and, as a consequence, $R_1\le L$. Apply now Lemma \ref{lema1-Picard-thm}.(\ref{L-le-R1}) to conclude that $L=R_1=r_C(T)$. $\hfill\qed$

\section{$\Delta$-convergence of Picard iterates}\label{cap:Picard-delta}

In 1976, Lim \cite{Lim76} introduced a concept of convergence in the general setting of metric spaces,
which is known as $\Delta$-convergence. Kuczumow \cite{Kuc80} introduced an identical notion of
convergence in Banach spaces, which he called {\em almost convergence}. As shown in \cite{KirPan08},
$\Delta$-convergence could be regarded, at least for CAT(0) spaces, as an analogue to the usual weak
convergence in Banach spaces. Jost \cite{Jos94} introduced a notion of weak convergence in CAT(0) spaces, which was
rediscovered by Esp\'{\i}nola and Fern\'{a}ndez-Le\'{o}n \cite{EspFer09}, who also proved that it is equivalent
to $\Delta$-convergence. We refer to \cite{Sos04} for other notions of weak convergence in geodesic spaces.

Let $(x_n)$ be a bounded sequence of a metric space $(X,d)$. We say that $(x_n)$
{\em $\Delta$-converges} to $x$ if $x$ is the unique asymptotic center of $(u_n)$ for every
subsequence $(u_n)$ of $(x_n)$. In this case, we write $x_n\conv{\Delta}x$ or $\Dlimn x_n=x$ and
we call $x$ the $\Delta$-limit of $(x_n)$.

Let $(X,d)$ be a metric space and $F\se X$ be a nonempty subset. A sequence $(x_n)$ in $X$ is said to
be \defnterm{Fej\' er monotone} with respect to $F$ if
\beq
d(p,x_{n+1})\le d(p,x_n)\quad \text{ for all } p\in F \text{ and }n\ge 0.
\eeq
Thus each point in the sequence is not further from any point in $F$ than its predecessor. Obviously,
any Fej\' er monotone sequence $(x_n)$ is bounded and moreover $(d(x_n,p))$ converges for every $p\in F$.

The following lemma is very easy to prove.

\blem\label{Delta-Fejer}
Let $(X,d)$ be a metric space, $F\se X$ be a nonempty subset and $(x_n)$ be Fej\' er monotone with respect to $F$. Then
\be
\item For all $p\in F$, $(d(p,x_n))$ converges and $r(p,(x_n))=\limn d(p,x_n)$.
\item Every subsequence $(u_n)$ of $(x_n)$ is Fej\' er monotone with respect to $F$ and for all
$p\in F$, $ r(p,(u_n))= r(p,(x_n))$. Hence, $r(F,(u_n))=r(F,(x_n))$ and $A(F,(u_n))=A(F,(x_n))$.
\item\label{as-cen-unic-Delta} If $A(F,(x_n))=\{x\}$ and $A((u_n))\se F$ for every subsequence
$(u_n)$ of $(x_n)$, then $(x_n)$ $\Delta$-converges  to $x\in F$.
\ee
\elem

Furthermore, one has the following result, whose proof is very similar to the one in strictly
convex Banach spaces. For the sake of completeness, we give it here.

\blem\label{Fix-closed-convex}
Let $C$ be a nonempty closed convex subset of a uniquely geodesic space $(X,d)$ and $T:C\to C$ be
nonexpansive. Then the set $Fix(T)$ of fixed points of $T$  is closed and convex.
\elem
\begin{proof}
The fact that $Fix(T)$ is closed is immediate from the continuity of $T$. We shall  prove its
convexity. Let $x,y\in Fix(T)$ be distinct  and $z\in[x,y]$. Then
\bua
d(x,y)&\le & d(x,Tz)+d(Tz,y)=d(Tx,Tz)+d(Tz,Ty)\leq d(x,z)+d(z,y)=d(x,y).
\eua
 Thus, $d(x,Tz)+d(Tz,y)=d(x,y)$, so that $Tz\in [x,y]$. We apply now
Lemma \ref{geodesic-prop}.(\ref{wz-xy-lies}) to get the following cases:
\bua
(i) & d(x,z)+d(z,Tz)=d(x,Tz)=d(Tx,Tz)\leq d(x,z),\\
(ii) & d(y,z)+d(z,Tz)=d(y,Tz)=d(Ty,Tz)\leq d(y,z).
\eua
In both cases, it follows that $Tz=z$.
\end{proof}

\bprop\label{Delta-conv-Picard-ne}
Let $(X,d,W)$ be a complete $UCW$-hyperbolic space, $C\se X$ be nonempty closed convex and
$T:C\to C$ be a nonexpansive mapping with $Fix(T)\neq\emptyset$. If $T$ is asymptotically regular at
$x\in C$, then the Picard iterate $(T^nx)$ $\Delta$-converges to a fixed point of $T$.
\eprop
\begin{proof}
By Lemma \ref{Fix-closed-convex}, the nonempty set $F:=Fix(T)$ is closed and convex. Furthermore,
one can see easily that $(T^nx)$ is Fej\' er monotone with respect to $F$ and, by
Theorem \ref{UCW-unique-ac}, $(T^nx)$ has a unique asymptotic center with respect to $F$.
Let $(u_n)$ be a subsequence of $(T^nx)$ and $u$ be its unique asymptotic center. Then
\bua
d(Tu,u_n)\leq d(Tu,Tu_n)+d(Tu_n,u_n)\leq d(u,u_n)+d(u_n,Tu_n),
\eua
so we can use Lemma \ref{UCW-useful-unique-as-center} to obtain that $Tu=u$, i. e.  $u\in F$. Apply
Lemma \ref{Delta-Fejer}.(\ref{as-cen-unic-Delta}) to get the conclusion.
\end{proof}

By \cite[Theorem 3.5]{Leu10} one can replace in the above theorem the assumption that $T$ has
fixed points with the equivalent one that $T$ has bounded orbits.

We get the following $\Delta$-convergence result for the Picard iteration of a firmly nonexpansive mapping.

\begin{theorem}\label{Delta-conv-Picard-firmly-ne}
Let $(X,d,W)$ be a complete $UCW$-hyperbolic space, $C\se X$ be nonempty closed convex and
$T:C\to C$ be a $\lambda$-firmly nonexpansive mapping for some $\lambda\in(0,1)$.
Assume that $Fix(T)\neq\emptyset$. Then for all $x$ in $C$, $(T^nx)$ $\Delta$-converges to a fixed point of
$T$.
\end{theorem}
\begin{proof}
Since $Fix(T)\neq\emptyset$, we get that $r_C(T)=0$, so, by Corollary \ref{rCT-as-reg}, that $T$ is
asymptotically regular. Apply now Proposition \ref{Delta-conv-Picard-ne}.
\end{proof}

\subsection{An application to a minimization problem}

Let $(X,d)$ be a complete CAT(0) space and $F:X\to(-\infty,\infty]$ be a proper, convex and lower
semicontinuous mapping. We shall apply Theorem \ref{Delta-conv-Picard-firmly-ne} to approximate the
minimizers of $F$, that is the solutions of  the minimization problem $\ds \min_{x\in X}F(x)$.

Let $\ds \argmin_{y\in X}F(y)=\{x\in X\mid  F(x)\leq F(y) \text{ for all }y\in X\}$ be the set of minimizers of
$F$. The following result  is a consequence of the definition of the resolvent and
Proposition \ref{local-global-dir-der}.

\begin{proposition}\label{pro:fixedpoint-min}
For all $\mu>0$,  the set $Fix(J_\mu)$ of fixed points of the resolvent associated with $F$ coincides with the
set $\ds \argmin_{y\in X}F(y)$ of minimizers of $F$.
\end{proposition}
\begin{proof} Let $\mu>0$. \\
$"\supseteq "$ If $\bar{x}$ is a minimizer of $F$, one gets that $\mu F(\bar{x})\leq\mu F(y)+d(\bar{x},y)^2$.
It follows that $\ds \bar{x}\in\argmin_{y\in X}\left\{\mu F(y)+d(\bar{x},y)^2\right\}$. By the definition of
$J_\mu$, it follows that $J_\mu(\bar{x})=\bar{x}$.\\
$"\subseteq "$ Assume that $J_\mu(\bar{x})=\bar{x}$, so $\mu F(\bar{x})\leq\mu F(y)+d(\bar{x},y)^2$ for all
$y\in X$. Let $\gamma:[0,c]\to X$ be a geodesic starting with $\bar{x}$. Then for all $t\in[0,c]$, one has that
$$\frac{F(\gamma(t))-F(\bar{x})}t\geq -\frac{d(\bar{x},\gamma(t))^2}{\mu t}=-\frac{t^2}{\mu}.$$
It follows that $D_\gamma F(\bar{x})\geq 0$, hence we can apply Proposition \ref{local-global-dir-der} to
conclude that $F(\bar{x})\leq F(y)$ for all $y\in X$.
\end{proof}

As the resolvent is a firmly nonexpansive mapping, one can apply Theorem
\ref{Delta-conv-Picard-firmly-ne} and the above result to obtain

\begin{corollary}
Assume that $F$ has a minimizer. Then for all $\mu>0$ and all $x\in X$, the Picard iterate
$(J_\mu^n(x))$ $\Delta$-converges to a minimizer of $F$.
\end{corollary}

We remark that a more general result was obtained recently by Ba\v{c}\'{a}k \cite{Bac12} using different
methods. Thus, Ba\v{c}\'{a}k obtained in the setting of CAT(0) spaces the following proximal point
algorithm:
if $F$ has minimizers, then  for all $x_0\in X$ and all sequences $(\lambda_n)$ divergent in sum, the sequence
\bua
x_{n+1}:=\argmin_{y\in X}\left(F(y)+\frac{1}{2\lambda_n}d(y,x_n)^2\right)
\eua
$\Delta$-converges to a minimizer of $F$.

\section{Effective rates of asymptotic regularity}\label{cap:rates-ucw-spaces}

As we have proved in Section \ref{cap:Picard-iterates}, any $\lambda$-firmly nonexpansive
mapping $T:C\to C$  defined on a nonempty subset $C$  of a  W-hyperbolic space $X$ is asymptotically regular,
provided $T$ has bounded orbits.

In this section we shall obtain, for $UCW$-hyperbolic spaces, a rate of asymptotic regularity of $T$, that is a
rate of convergence of the sequence $(d(T^nx,T^{n+1}))$ towards $0$. The methods of proof are inspired by
those used by Kohlenbach \cite{Koh03} and the second author \cite{Leu07} for computing rates of
asymptotic regularity for the Krasnoselski-Mann iterations of nonexpansive mappings in uniformly convex
Banach spaces and $UCW$-hyperbolic spaces.

For $x\in C$ and $b,\eps>0$, let us denote
\[Fix_{\varepsilon}(T,x,b):=\{y\in C\mid d(y,x)\leq b \text{~and~} d(y,Ty)<\eps\}.\]
If $Fix_{\varepsilon}(T,x,b)\ne \emptyset$ for all $\eps>0$, we say that $T$
{\em has approximate fixed points} in a $b$-neighborhood of $x$.

\bthm\label{rate-as-reg-ucw}
Let $b>0,\lambda\in(0,1)$ and $\eta:(0,\infty)\times(0,2]\rightarrow(0,1]$ be a mapping that decreases with
$r$ for fixed $\eps$. Then for all UCW-hyperbolic
spaces $(X,d,W,\eta)$, nonempty subsets $C\se X$, $\lambda$-firmly nonexpansive mappings $T:C\to C$ and
all $x\in C$ such that  $T$ has approximate fixed points in a
$b$-neighborhood of $x$, the following holds:
\begin{equation}\label{eq:R2}
\forall \eps>0\,\forall\,n\geq \Phi(\eps,\eta,\lambda,b)\, \big(d(T^nx,T^{n+1}x)\leq\eps\big),
\end{equation}
where
\beq
\Phi(\eps,\eta,\lambda,b):=\begin{cases}\ds
\left[\frac{b+1}{\ds \eps\,\lambda\,(1-\lambda)\,\eta\left(b+1,\frac{\eps}{b+1}\right)}\right] & \text
{for  }\eps<2b, \\
 0 & \text{otherwise}.\label{def-Phi}
\end{cases}
\eeq
\ethm

\begin{remark}\label{tilde-eta}
If, moreover,  $\eta(r,\eps)$ can be written as $\eta(r,\eps)=\eps\cdot\tilde{\eta}(r,\eps)$ such that
$\tilde{\eta}$ increases with $\eps$ (for a fixed $r$), then the bound $\Phi(\eps,\eta,\lambda,b)$ can be
replaced for $\eps<2b$ by
\beq
\tilde{\Phi}(\eps,\eta,\lambda,b)=\left[\frac{b+1}{\ds \eps\,\lambda\,(1-\lambda)\,\tilde{\eta}
\left(b+1,\frac{\eps}{b+1}\right)}\right]
\eeq
\end{remark}

Before proving the above results, let us give two consequences.

\bcor\label{rate-as-reg-bounded-C}
Let $b,\lambda,\eta$ be as in the hypothesis of Theorem \ref{rate-as-reg-ucw}.
Then for all UCW-hyperbolic spaces $(X,d,W,\eta)$, bounded subsets $C\se X$ with diameter $d_C\leq b$,
$\lambda$-firmly nonexpansive mappings $T:C\to C$ and all $x\in C$,
\bua
\forall \eps>0\,\forall\,n\geq \Phi(\eps,\eta,\lambda,b)\, \big(d(T^nx,T^{n+1}x)\leq\eps\big),
\eua
where $\Phi(\eps,\eta,\lambda,b)$ is given by \eqref{def-Phi}.
\ecor
\begin{proof}
If $C$ is bounded, then $T$ is asymptotically regular by Corollary \ref{T-bounded-orbits-as-reg}. Hence, for
all  $b\ge d_C$, $T$ has approximate fixed points in a $b$-neighborhood of $x$ for all $x\in C$.
\end{proof}

Thus, for bounded $C$, we get that $T$ is asymptotically regular  with a rate
$\Phi(\varepsilon,\eta,\lambda,b)$ that only depends on $\varepsilon$, on $X$ via the monotone modulus of
uniform convexity $\eta$, on $C$ via an upper bound $b$ on its diameter $d_C$ and on the mapping
$T$ via $\lambda$. The rate of asymptotic regularity is uniform in the starting point $x\in C$ of the iteration
and other data related with $X,C$ and $T$.

As we have remarked in Section \ref{cap:geodesic}, CAT(0) spaces are $UCW$-hyperbolic spaces with a
quadratic (in $\eps$) modulus of uniform convexity $\displaystyle \eta(\varepsilon)=\frac{\varepsilon^2}{8}$, which has
the form required in Remark \ref{tilde-eta}. As an immediate consequence,  we get a quadratic (in $1/\varepsilon$) rate of asymptotic regularity
in the case of CAT(0) spaces.

\bcor\label{rate-as-reg-CAT0}
Let $b>0$ and $\lambda\in (0,1)$. Then for all CAT(0) spaces $X$, bounded subsets $C\se X$ with diameter
$d_C\leq b$, $\lambda$-firmly nonexpansive mappings $T:C\to C$ and $x\in C$, the following holds
\bua
\forall \varepsilon >0\, \forall n\ge \Psi(\eps,\lambda,b)\,
\big( d(T^nx,T^{n+1}x)\leq\varepsilon\big),
\eua
where
\bua
\Psi(\eps,\lambda,b):=\begin{cases}\ds
\left[\frac{8(b+1)}{\lambda\,(1-\lambda)}\cdot \frac{1}{\eps^2} \right] & \text
{for  }\eps<2b, \\
 0 & \text{otherwise}.
\end{cases}
\eua
\ecor

\subsection{Proof of Theorem \ref{rate-as-reg-ucw} and Remark \ref{tilde-eta}}

One can easily see that $d(T^nx,T^{n+1}x)\leq 2b$ for all $n\in\N$, hence the case $\eps\geq 2b$ follows.

Assume now that $\eps<2b$ and denote
\beq
N:=\Phi(\eps,\eta,\lambda,b)=\left[\frac{b+1}{\ds \eps\,\lambda\,(1-\lambda)\,\eta\left(b+1,\frac{\eps}{b+1}\right)}\right].
\eeq
Let $\delta>0$ be such that $\ds \delta<\frac{1}{4(N+1)}$, so that $\ds (N+1)\delta<\frac{1}{4}<1$.
By hypothesis, there exists $y\in C$ satisfying
\beq
d(x,y)\le b \quad\text{and}\quad d(y,Ty)<\delta.
\eeq
We shall prove that there exists $n\le N$ such that $d(T^nx,T^{n+1}x)\leq\eps$ and apply the fact that $(d(T^nx,T^{n+1}x))$
is nonincreasing to get the conclusion. Assume by contradiction that $d(T^nx,T^{n+1}x)>\eps$ for all $n=0,\ldots,N$.
In the sequel, we fix such an $n$. For simplicity we shall use the notation
\beq
r_n:= d(T^nx,y)+d(y,Ty).
\eeq
One gets by an easy induction that
\bua
r_n &\leq & d(x,y)+(n+1)\,d(y,Ty)\leq b+(N+1)\delta < b+1.
\eua
Since
\bua
d(T^{n+1}x,y)\leq d(T^{n+1}x,Ty)+d(Ty,y)\leq r_n, & d(T^nx,y)\le r_n,
\eua
\bua
d(T^nx,Ty)\le d(T^nx,y)+d(y,Ty)=r_n, & d(T^{n+1}x,Ty)\le r_n
\eua
and $d(T^nx,T^{n+1}x)>\eps$,
we can apply twice Lemma \ref{eta-prop-1}.(\ref{eta-monotone-s-geq-r})  with $r:=r_n$ and $s:=b+1$ to get that
\bua
d((1-\lambda)T^nx\oplus\lambda
T^{n+1}x,y) & \leq & \left(1-2\lambda(1-\lambda)\eta\left(b+1,\frac{\eps}{b+1}\right)\right)r_n,\\
d((1-\lambda)T^nx\oplus\lambda T^{n+1}x,Ty) &\leq & \left(1-2\lambda(1-\lambda)\eta\left(b+1,\frac{\eps}{b+1}\right)\right)r_n.
\eua
As $T$ is $\lambda$-firmly nonexpansive, it follows that
\bua
d(T^{n+1}x,Ty) &\leq & d((1-\lambda)T^nx\oplus\lambda T^{n+1}x,(1-\lambda)y\oplus\lambda Ty)\\
&\leq & (1-\lambda)\,d((1-\lambda)T^nx\oplus\lambda
T^{n+1}x,y)+\\
&& + \lambda\,d((1-\lambda)T^nx\oplus\lambda T^{n+1}x,Ty) \quad \text{by (W1)}\\
&\leq & \left(1-2\lambda(1-\lambda)\eta\left(b+1,\frac{\eps}{b+1}\right)\right)r_n\\
&=& d(T^nx,y)+d(y,Ty)-2r_n\lambda(1-\lambda)\eta\left(b+1,\frac{\eps}{b+1}\right)\\
&\leq & d(T^nx,y)+\delta-\eps\lambda(1-\lambda)\eta\left(b+1,\frac{\eps}{b+1}\right),
\eua
since $d(y,Ty)\le\delta$ and
\bua
\frac{\eps}2 &< & \frac12 d(T^nx,T^{n+1}x)\le \frac12\left(d(T^nx,y)+d(y,Ty)+d(Ty,T^{n+1}x)\right)\le r_n.
\eua
Using now the fact that $d(T^{n+1}x,y)\le d(T^{n+1}x,Ty)+d(y,Ty)$, we get that
\beq
d(T^{n+1}x,y) \le d(T^nx,y)+2\delta-\eps\lambda(1-\lambda)\eta\left(b+1,\frac{\eps}{b+1}\right). \label{proof-eff-rate-ineq-1}
\eeq
Adding (\ref{proof-eff-rate-ineq-1}) for $n=0,\ldots, N$, it follows that
\bua
d(T^{N+1}x,y) &\le & d(x,y) + 2(N+1)\delta-(N+1)\eps\lambda(1-\lambda)\eta\left(b+1,\frac{\eps}{b+1}\right)\\
&\leq & b+\frac{1}{2}-(N+1)\eps\lambda(1-\lambda)\eta\left(b+1,\frac{\eps}{b+1}\right)\\
&\leq & b+\frac{1}{2}-(b+1)<0,
\eua
that is a contradiction.$\qed$\\[0.2cm]

To prove Remark \ref{tilde-eta}, observe that, by denoting
\bua
N:=\tilde{\Phi}(\eps,\eta,\lambda,b)=\left[\frac{b+1}{\ds \eps\,\lambda\,(1-\lambda)\,\tilde{\eta}\left(b+1,\frac{\eps}{b+1}\right)}\right],
\eua
and following the proof above with $r_n$ instead of $b+1$ we obtain
\bua
d(T^{n+1}x,Ty) &\leq &  d(T^nx,y)+\delta-r_n\lambda(1-\lambda)\eta\left(r_n,\frac{\eps}{r_n}\right)\\
&\leq &  d(T^nx,y)+\delta-r_n\lambda(1-\lambda)\eta\left(b+1,\frac{\eps}{r_n}\right)\\
&& \text{since } \eta \text{ is monotone}\\
&=& d(T^nx,y)+\delta-\eps\lambda(1-\lambda)\tilde{\eta}\left(b+1,\frac{\eps}{r_n}\right)\\
&\leq & d(T^nx,y)+\delta-\eps\lambda(1-\lambda)\tilde{\eta}\left(b+1,\frac{\eps}{b+1}\right)\\
&& \text{since } \tilde{\eta} \text{ increases with }\eps.
\eua
Follow now the proof above to get the conclusion. $\hfill\Box$

\mbox{ }

\section*{Acknowledgements:}

David  Ariza-Ruiz was supported by Junta de Andalucia, Grant FQM-3543. Part of his research was carried out
while visiting the Simion Stoilow Institute of Mathematics of the Romanian Academy.\\[1mm]
Lauren\c tiu Leu\c stean was supported by a grant of the Romanian National Authority for Scientific
Research, CNCS - UEFISCDI, project number PN-II-ID-PCE-2011-3-0383.\\[1mm]
Genaro L\'{o}pez-Acedo was partially supported by DGES, Grant MTM2009-13997-C02-01 and Junta de Andalucia,
Grant FQM-127.\\[1mm]
The authors gratefully acknowledge the anonymous reviewer for helpful comments and suggestions.


\begin{thebibliography}{00}

\bibitem{Bac12}
M. Ba\v{c}\'{a}k,  The proximal point algorithm in metric spaces, Israel J. Math., 2012, doi: 10.1007/s11856-012-0091-3.

\bibitem{BauMofWan12}
H.H. Bauschke, S.M. Moffat, X. Wang, Firmly nonexpansive mappings and maximally monotone operators:
correspondence and duality, Set-Valued Variational Anal. 20 (2012), 131-153.

\bibitem{BreCraPaz70} H.  Br\'{e}zis,  M.G.  Crandall, A. Pazy, Perturbations of nonlinear maximal monotone sets in
Banach spaces, Comm. Pure Appl Math. 23 (1970), 123-144.

\bibitem{BriHae99}
M. Bridson, A. Haefliger, Metric spaces of non-positive curvature, Grundlehren der Mathematischen
Wissenschaften 319, Springer-Verlag, Berlin, 1999.

\bibitem{Bro67}
F.E. Browder, Convergence theorems for sequences of nonlinear operators in Banach spaces,  Math. Zeitschrift
100 (1967), 201-225.

\bibitem{Bru73}
R. E. Bruck, Nonexpansive projections on subsets of Banach spaces, Pacific J. Math.  47 (1973), 341-355.

\bibitem{BruRei77}
R. E. Bruck, S. Reich, Nonexpansive projections and resolvents of accretive operators in Banach spaces, Houston J. Math. 3 (1977),
459-470.

\bibitem{BruTit72}
M. Bruhat, J. Tits,  Groupes r\'{e}ductifs sur un corps local. I. Donn\'{e}es radicielles valu\'{e}es,
Inst. Hautes \'{E}tudes Sci. Publ. Math. 41 (1972), 5-251.

\bibitem{Bus48}
H. Busemann, Spaces with nonpositive curvature,  Acta Math. 80 (1948), 259-310.

\bibitem{Bus55}
H. Busemann, The geometry of geodesics, Pure Appl. Math. 6, Academic Press, New York, 1955.

\bibitem{Ede72}
M. Edelstein, The construction of an asymptotic center with a fixed-point property, Bull. Amer. Math. Soc.
78 (1972), 206-208.

\bibitem{Ede74}
M. Edelstein, Fixed point theorems in uniformly convex Banach spaces, Proc. Amer. Math. Soc. 44 (1974), 369-374.

\bibitem{EspFer09}
R. Esp\'{\i}nola, A. Fern\'andez-Le\'on, CAT($\kappa$)-spaces, weak convergence and fixed points,
J. Math. Anal. Appl.  353 (2009), 410-427.

\bibitem{Fek23}
M. Fekete, \"{U}ber die Verteilung der Wurzeln bei gewissen algebraischen Gleichungen mit ganzzahligen
Koeffizienten, Math. Zeitschrift 17 (1923), 228-249.

\bibitem{FerPerNem05}
O.P. Ferreira, L.R.  Lucambio P\'erez, S.Z. N\'emeth, Singularities of monotone vector fields and an
extragradient-type algorithm, J. Global Optim. 31 (2005),  133-151.

 \bibitem{GarRei06}
 J. Garc\'{\i}a-Falset, S. Reich, Zeroes of accretive operators and the asymptotic behavior of nonlinear semigroups,
 Houston J. Math. 32 (2006), 1197-1225.

\bibitem{GoeKir83}
K. Goebel, W. A. Kirk, Iteration processes for nonexpansive mappings, in:  S.P. Singh, S. Thomeier,
B. Watson (eds.), Topological methods in nonlinear functional analysis. Proceedings of the special session on
fixed point theory and applications held during the 86th summer meeting of the American Mathematical Society at
the University of Toronto, Toronto, Ont., August 21-26, 1982, Contemp. Math. 21, Amer. Math. Soc., Providence,
RI, 1983, 115-123.

\bibitem{GoeRei84}
K. Goebel, S. Reich, Uniform convexity, hyperbolic geometry, and nonexpansive mappings, Marcel Dekker, Inc.,
New York and Basel, 1984.

\bibitem{Jos94}
J. Jost, Equilibrium maps between metric spaces, Calc. Var. Partial Diff. Equations 2 (1994), 173-204.

\bibitem{Jos95}
J. Jost, Convex functionals and generalized harmonic maps into spaces of nonpositive curvature,
Comment. Math. Helv. 70 (1995), 659-673.

\bibitem{Kir82}
W.A. Kirk, Krasnosel'skii iteration process in hyperbolic spaces, Numer. Funct. Anal. Optimiz.
4 (1982), 371-381.

\bibitem{KirPan08}
W.A. Kirk, B. Panyanak, A concept of convergence in geodesic spaces, Nonlinear Anal. 68 (2008), 3689-3696.

\bibitem{Koh03}
U. Kohlenbach, Uniform asymptotic regularity for Mann iterates, J. Math. Anal. Appl. 279 (2003), 531-544.

\bibitem{Koh05}
U. Kohlenbach, Some logical metatheorems with applications in functional analysis,
Trans. Amer. Math. Soc. 357 (2005), 89-128.

\bibitem{Koh08-book}
U. Kohlenbach, Applied proof theory: Proof interpretations and their use in mathematics,
Springer Monographs in Mathematics, Springer-Verlag, Berlin-Heidelberg, 2008.

\bibitem{KohLeu10}
U. Kohlenbach, L. Leu\c{s}tean,  Asymptotically nonexpansive mappings in uniformly convex hyperbolic spaces,
J. European Math. Soc. 12 (2010), 71-92.

\bibitem{KopRei09}
E. Kopeck\'{a}, S. Reich, Asymptotic behavior of resolvents of coaccretive operators in the Hilbert ball,
Nonlinear Anal. 70 (2009), 3187-3194.

\bibitem{Kuc80}
T. Kuczumow, An almost convergence and its applications, Ann. Univ. Mariae Curie-Sklodowska Sect. A 32 (1978),
79-88.

\bibitem{Leu07}
L. Leu\c{s}tean, A quadratic rate of asymptotic regularity for CAT(0) spaces, J. Math. Anal. Appl. 325 (2007),
386-399.

\bibitem{Leu10}
L. Leu\c{s}tean, Nonexpansive iterations in uniformly convex $W$-hyperbolic spaces, in: A. Leizarowitz,
B. S. Mordukhovich, I. Shafrir, A. Zaslavski (Editors): Nonlinear Analysis and Optimization I: Nonlinear
Analysis, Contemporary Mathematics, Vol. 513 (2010), AMS, 193-209.

\bibitem{LiLopNar09}
C. Li, G.  L\'{o}pez-Acedo, V.  Mart\'in-M\'arquez, Monotone vector fields and the proximal point algorithm on
Hadamard manifolds, J. London Math. Soc. 79 (2009), 663-683.

\bibitem{Lim76}
T.C. Lim, Remarks on some fixed point theorems, Proc. Amer. Math. Soc. 60 (1976), 179-182.

\bibitem{May98}
U.F. Mayer, Gradient flows on nonpositively curved metric spaces and harmonic maps, Comm. Anal. Geom. 6 (1998),
199-253.

\bibitem{Min62}
G.J. Minty, Monotone (nonlinear) operators in Hilbert space, Duke Math. J. 29 (1962), 341-346.

\bibitem{Mor65}
J.-J. Moreau, Proximit\'{e} et dualit\'{e} dans un space Hilbertien, Bull. Soc. Math. France, 93(1965), 273-299.

\bibitem{NevRei79}
O. Nevanlinna, S. Reich, Strong convergence of contractions semigroups and of iterative methods for accretive operators
in Banach spaces, Israel J. Math. 32 (1979), 44-58.

\bibitem{Nic12}
A. Nicolae, Asymptotic behavior of firmly nonexpansive and averaged mappings in geodesic spaces, preprint, 2012.

\bibitem{Pap05}
A. Papadopoulos, Metric spaces, convexity and nonpositive curvature, IRMA Lectures in Mathematics and
Theoretical Physics 6, European Mathematical Society, 2005.

\bibitem{Rei77}
S. Reich, Extension problems for accretive sets in Banach spaces, J. Functional Anal. 26 (1977), 378-395.

\bibitem{ReiSha87}
S. Reich, I. Shafrir, The asymptotic behavior of firmly nonexpansive mappings, Proc. Amer. Math. Soc.
101 (1987), 246-250.


\bibitem{ReiSha90}
S. Reich, I. Shafrir, Nonexpansive iterations in hyperbolic spaces, Nonlinear Anal.  15 (1990), 537-558.

\bibitem{Roc76}
R.T. Rockafellar, Monotone operators and the proximal point algorithm, SIAM J. Control Optim. 14 (1976),
877-898.

\bibitem{Sma91}
R. Smarzewski, On firmly nonexpansive mappings, Proc. Amer. Math. Soc. 113 (1991), 723-725.

\bibitem{Sos04}
E.N. Sosov,  On analogues of weak convergence in a special metric space, Izv. Vyssh. Uchebn. Zaved. Mat.
5 (2004), 84-89. Translation in Russian Math. (Iz.VUZ) 48 (2004), 79-83.

\bibitem{Sto11}
I. Stojkovic, Geometric approach to evolution problems in metric spaces, PhD thesis, 2011,
\url{http://www.math.leidenuniv.nl/scripties/PhDThesisStojkovic.pdf}.

\bibitem{Tak70}
W. Takahashi, A convexity in metric space and nonexpansive mappings I, Kodai Math. Sem. Rep. 22 (1970),
142-149.

\bibitem{Xu01}
H.K. Xu, Strong asymptotic behavior of almost orbits  of nonlinear semigroups, Nonlinear Anal.  46 (2001), 135-151.

\end{thebibliography}
\end{document}